\documentclass[oneside]{amsart}
\usepackage[utf8]{inputenc}
\usepackage{amsthm,geometry}
\theoremstyle{plain}
\newtheorem{theorem}{Theorem}[section]
\newtheorem{introthm}{Theorem}
\newtheorem{lemma}[theorem]{Lemma}
\newtheorem{corollary}[theorem]{Corollary}
\newtheorem{proposition}[theorem]{Proposition}
\theoremstyle{definition}
\newtheorem{definition}[theorem]{Definition}
\theoremstyle{remark}
\newtheorem{remark}[theorem]{Remark}
\newtheorem{example}[theorem]{Example}
\usepackage{mathtools,amssymb,mleftright,stmaryrd,scalerel,tikz,textcomp,booktabs}
\mleftright

\makeatletter
\AtBeginDocument{%
 \let\glb@currsize\relax
}
\makeatother

\usepackage[scr=boondox]{mathalfa}
\usepackage{xparse}
\ExplSyntaxOn
\NewDocumentCommand{\definealphabet}{mmmm}
 {%
  \int_step_inline:nnn { `#3 } { `#4 }
   {
    \cs_new_protected:cpx { #1 \char_generate:nn { ##1 }{ 11 } }
     {
      \exp_not:N #2 { \char_generate:nn { ##1 } { 11 } }
     }
   }
 }
\ExplSyntaxOff
\definealphabet{bb}{\mathbb}{A}{Z}
\definealphabet{cal}{\mathcal}{A}{Z}
\definealphabet{frak}{\mathfrak}{A}{z}
\definealphabet{bf}{\mathbf}{A}{z}
\DeclareMathOperator{\Supp}{supp}

\let\opn\operatorname
\newcommand{\lto}{\longrightarrow}

\newcommand{\Lpha}{\bbL_p^{\opn{ha}}}
\newcommand{\Newt}[1]{\operatorname{\mathscr{N\kern-3pt e\kern-1pt w\kern-1pt t}}\left(#1\right)}
\usepackage[pdfencoding=auto,psdextra]{hyperref}
\usepackage[nameinlink]{cleveref}
\Crefname{conjecture}{Conjecture}{Conjectures}
\Crefname{lemma}{Lemma}{Lemmas}
\Crefname{definition}{Definition}{Definitions}
\Crefname{remark}{Remark}{Remarks}
\Crefname{proposition}{Proposition}{Propositions}
\Crefname{corollary}{Corollary}{Corollaries}
\Crefname{algorithm}{Algorithm}{Algorithms}
\Crefname{example}{Example}{Examples}
\Crefname{proof}{Proof}{Proofs}
\usepackage[url=false,backend=biber,style=ext-alphabetic,hyperref=true,giveninits=true,maxbibnames=99]{biblatex}
\usepackage{xcolor}
\usepackage{enumitem,float,algorithm,algpseudocode,cases,adjustbox,caption}
\usepackage[all]{xy}
\newtheorem*{theorem*}{Theorem}

\newtheorem*{proposition*}{Proposition}

\numberwithin{equation}{section}
\numberwithin{figure}{section}
\newlist{propenum}{enumerate}{1}
\setlist[propenum]{label=(\arabic*), ref=\theproposition~(\arabic*)}
\crefalias{propenumi}{proposition}
\newlist{rmkenum}{enumerate}{1}
\setlist[rmkenum]{label=(\arabic*), ref=\theproposition~(\arabic*)}
\crefalias{rmkenumi}{remark}
\newlist{lemenum}{enumerate}{1}
\setlist[lemenum]{label=(\arabic*), ref=\thelemma~(\arabic*)}
\crefalias{lemenumi}{lemma}
\newlist{thmenum}{enumerate}{1}
\setlist[thmenum]{label=(\arabic*), ref=\thetheorem~(\arabic*)}
\crefalias{thmenumi}{theorem}
\Crefformat{enumi}{#2\textup{(#1)}#3}
\algnewcommand\algorithmicinput{\textbf{INPUT:}}
\algnewcommand\INPUT{\item[\algorithmicinput]}
\algnewcommand\algorithmicoutput{\textbf{OUTPUT:}}
\algnewcommand\OUTPUT{\item[\algorithmicoutput]}
\usepackage{microtype}
\allowdisplaybreaks

\usepackage{orcidlink}

\begin{document}
\title{Hyper-algebraic invariants of $p$-adic algebraic numbers}
\author{Shanwen Wang\orcidlink{0000-0003-0228-1208}}
\address{School of Mathematics, Renmin University of China, No. 59 Zhongguancun Street, Haidian District, Beijing, 100872, China}
\email{s\_wang@ruc.edu.cn}
\author{Yijun Yuan\orcidlink{0000-0001-6571-6980}}
\address{Institute for Theoretical Sciences, Westlake University, No. 600 Dunyu Road, Sandun town, Xihu district, Hangzhou, Zhejiang, 310030, China}
\email{941201yuan@gmail.com}
\begin{abstract}
Let $p\geq 3$ be a prime. The hyper-algebraic elements in the $p$-adic Mal'cev-Neumann field $\bbL_p$ form an algebraically closed subfield $\Lpha$. In this article, we clarify the relations among the fields $\Lpha$, $\overline{\bbQ}_p$ and $\bbC_p$. We introduce two arithmetic invariants (hyper-tame index and hyper-inertia index) of hyper-algebraic elements and study the relation between these invariants and classical arithmetic invariants of $p$-adic algebraic numbers. Finally, we give a criterion for hyper-algebraic elements to be tamely ramified over $\bbQ_p$.
\end{abstract}
\maketitle

\section{Introduction}
Let $p\geq 3$ be a prime throughout this article. The $p$-adic Mal'cev-Neumann field $\bbL_p\coloneqq W(\overline{\bbF}_p)(\!(p^\bbQ)\!)$, constructed in \cite{Poonen1993}, is the unique minimal spherically complete extension of the field $\bbC_p$ of $p$-adic complex numbers. An element $f\in \bbL_p$ can be written uniquely in the form
\[f=\sum_{q\in \bbQ}[r_q]p^q, \text{ where }[\cdot]\colon\overline{\bbF}_p\lto W(\overline{\bbF}_p) \text{ is the Teichm\"{u}ller character}\]
and $\Supp(f)=\{q\in \bbQ\colon r_q\neq 0\}$ is a well-ordered subset of $(\bbQ,\leq)$. Thus, an element $f=\sum_{q\in \bbQ}[r_q]p^q$ of $\bbL_p$ is completely determined by its support $\Supp(f)$ and the set $\{r_q\}_{q\in\bbQ}$ of its coefficients.

The spherically complete condition is crucial in non-Archimedean functional analysis (see \cite[Proposition 9.2]{Schneider2002} for a concrete example). In arithmetic geometry, it also serves as an intermediate hypothesis in Scholze and Weinstein's classification of $p$-divisible groups over the ring $\mathcal{ O}_{\bbC_p}$ of integers of $\bbC_p$ (cf. \cite[Proposition 5.2.5]{Scholze2013}). Besides the importance of spherical completeness, it is surprising that not much arithmetic of $\bbL_p$ is investigated. We summarize several results from the literature:
\begin{enumerate}
	\item In \cite{lampertAlgebraicPadicExpansions1986}, Lampert introduced the notion of $p$-adic Mal'cev-Neumann series. In particular, he proved that the elements in $\bbL_p$, satisfying that the accumulating points of the support are all rational, form an algebraically closed field (cf. \cite[Theorem 2]{lampertAlgebraicPadicExpansions1986}). %
	\item In \cite{Poonen1993}, Poonen gave a rigorous construction of the field $\bbL_p$ and systematically studied various aspects of this field. In particular, a necessary condition for an element of $\bbL_p$ to be algebraic over $\bbQ_p$, which is claimed by Lampert in \cite[282]{lampertAlgebraicPadicExpansions1986}, is proved in \cite[Corollary 8]{Poonen1993}.
	\item Based on an idea of Lampert, Kedlaya proposed a transfinite Newton algorithm (cf. \cite[Proposition 1]{kedlayaPowerSeriesPAdic2001}) to prove the algebraic closeness of $\bbL_p$ effectively, which is extracted in \cite[Algorithm 1]{WangYuan2021}.
	\item In \cite[Theorem 13.4]{kedlayaAlgebraicityGeneralizedPower2017b}, Kedlaya gave a necessary and sufficient condition for an element in $\bbL_p$ to be a $p$-adic complex number, in terms of the so-called ``$p$-quasi-automatic elements''.
	\item The truncated expansions of roots of unity in $\bbL_p$ are studied in \cite[Theorem 3.3]{WangYuan2021} and \cite[Theorem 1.6]{OurArxiv}. Based on these results, the uniformizers of the $p$-adic false Tate curve extensions $\bbK_p^{m,n}\coloneqq \bbQ_p\left(\zeta_{p^m},p^{1/p^n}\right)$ for $(m,n)\in\left(\{2\}\times\bbZ_{\geq 1}\right)\cup \left(\bbZ_{\geq 3}\times\{1\}\right)$ are constructed (cf. \cite{WangYuan2021,wangUniformizerFalseTate2024}).
	\item On the field $\bbL_p$, we can define a canonical Frobenius map
	      by the formula
	      $$\varphi\colon \sum_{q\in\bbQ}[r_q]p^q\longmapsto \sum_{q\in\bbQ}[r_q^p]p^q.$$
	      In \cite{efimov2024hahnwittseriesgeneralizations}, Efimov proved that $\varphi$ acts on the systems of $p^n$-th roots of unity by taking inverse. Note that one can view the complex conjugation as the Frobenius automorphism of $\bbC$, and the result of Efimov justifies that the Frobenius $\varphi$ can be viewed as the complex conjugation on $\bbL_p$.

\end{enumerate}
The purpose of this article is to answer several natural questions concerning the arithmetic of the field $\bbL_p$, which we make precise in the following.%

\subsection{Criterion of algebraicity}

By \cite[282]{lampertAlgebraicPadicExpansions1986} and \cite[Corollary 8]{Poonen1993}, if $f\in\bbL_p$ is algebraic over $\bbQ_p$, then it satisfies the following conditions:
\begin{enumerate}
	\item there exists a positive integer $N$ such that $\opn{supp}(f)\subseteq \frac{1}{N}\bbZ[1/p]$;
	\item there exists a positive integer $k$ such that $r_q\in\bbF_{p^k}$ for all $q\in\opn{supp}(f)$.
\end{enumerate}
An element $f\in \bbL_p$ satisfying the above conditions is called \textit{hyper-algebraic}. The set $\Lpha$ of hyper-algebraic elements in $\bbL_p$ forms an algebraically closed field containing $\bbQ_p$. As a result, all $p$-adic algebraic numbers are hyper-algebraic, i.e. $\overline{\bbQ}_p\subseteq \Lpha$. Our first result is a clarification of relations among the fields $\Lpha$, $\overline{\bbQ}_p$ and $\bbC_p$:
\begin{introthm}[cf. \Cref{prop:54163}]
	The field $\Lpha$ is strictly larger than $\overline{\bbQ}_p$, and it is neither complete nor a subfield of $\bbC_p$.
\end{introthm}

For a hyper-algebraic element $\alpha\in\bbL_p^{\opn{ha}}$, we introduce two new invariants of $\alpha$, called the hyper-tame index $\frakT_\alpha$ and hyper-inertia index $\frakF_\alpha$, defined to be the minimal integers $N$ and $k$ in the conditions given by Poonen respectively. %
For a $p$-adic algebraic number $\alpha\in\overline{\bbQ}_p$, its hyper-algebraic invariants $\frakT_\alpha$ and $\frakF_\alpha$ are closely related to its usual arithmetic invariants.
\begin{introthm}[cf. \Cref{prop:47781}, \Cref{prop:57726},\Cref{thm:53765}]
	Let $\alpha$ be a $p$-adic algebraic number.
	\begin{enumerate}
		\item The hyper-algebraic invariants $\frakT_\alpha$ and $\frakF_\alpha$ do not exceed $[\bbQ_p(\alpha)\colon\bbQ_p]$;
		\item Suppose $\bbQ_p(\alpha)/
			      \bbQ_p$ is an abelian extension of degree $n$. Denote by $\bff_{\bbQ_p(\alpha)}$ the local conductor of $\bbQ_p(\alpha)$ over $\bbQ_p$. Then
		      \begin{enumerate}
			      \item If $\bff_{\bbQ_p(\alpha)}=0$, then $\frakT_\alpha=1$ and $\frakF_\alpha=n$.
			      \item If $\bff_{\bbQ_p(\alpha)}\geq 1$, then $\frakT_\alpha\mid p-1$ and
			            $$\frakF_\alpha\mid\begin{cases}
					            \opn{lcm}(2,n),                                              & \text{ if }\bff_{\bbQ_p(\alpha)}=1,2;   \\
					            \opn{lcm}\left(2\cdot p^{ \bff_{\bbQ_p(\alpha)}-1},n\right), & \text{ if }\bff_{\bbQ_p(\alpha)}\geq 3.
				            \end{cases}.$$
		      \end{enumerate}
	\end{enumerate}
\end{introthm}
\begin{remark}
	The proof of this result is based on our computation of the truncated expansion of $\zeta_{p^n}$ (cf. \cite{WangYuan2021,OurArxiv}, and also see \Cref{eg:29565} for the precise formula).
\end{remark}

\begin{remark}\leavevmode
	For $\alpha\in \bbL_p$, we denote by $[C_{\frac{1}{p-1}}(\alpha)]$ the coefficient of index $\frac{1}{p-1}$ of the canonical expansion of $\alpha$.
	Based on the truncated expansion of $\zeta_{p^n}$ (cf. \Cref{eg:29565}), we conjecture that for any integer $n\geq 2$ and $p^n$-th primitive root of unity $\zeta_{p^n}$, there exists another $p^n$-th primitive root of unity $\zeta_{p^n}^\prime$ with $C_{\frac{1}{p-1}}(\alpha)=0$ such that $\zeta_{p^n}^{p^{n-1}}=\left(\zeta_{p^n}^\prime\right)^{p^{n-1}}$.%

	If this conjecture holds\footnote{We notice that in a recent preprint (cf. \cite{efimov2024hahnwittseriesgeneralizations}), Efimov claimed (ibid., Section 2) that his main theorem (ibid.) implies $\frakF_{\zeta_{p^n}}=2$ for every $n\geq 1$. With his result, we can bypass the aforementioned conjecture.}, then $\frakF_{\zeta_{p^n}}=2$ for every $n\geq 2$, and consequently $\frakF_\alpha$ divides $\opn{lcm}(2,n)$ for all ramified cases in the above theorem. See the proof of \Cref{prop:14697} for more details. Note that this conjecture is true when $n=2$ (cf. \Cref{lem:2580}).
\end{remark}
Our third result is to give a criterion for hyper-algebraic element to be tamely ramified over $\bbQ_p$:
\begin{introthm}[cf. \Cref{thm:54918new}]
	Let $\alpha\in\Lpha$ be a hyper-algebraic element in $\bbL_p$. Then $\bbQ_p(\alpha)$ is tamely ramified over $\bbQ_p$ if and only if $\Supp(\alpha)\subseteq\frac{1}{\frakT_\alpha}\bbZ$. In this situation, we have $\frakT_\alpha=\frake_\alpha$, $\frakf_\alpha\mid \frakF_\alpha$ and $\frakF_\alpha \mid c$, where $c\coloneqq \opn{ord}_{\opn{lcm}(\frake_\alpha,p^{\frakf_\alpha}-1)}p$ and $\frakf_\alpha$ (resp. $\frake_\alpha$) is the inertia degree (resp. the ramification index) of the extension $\bbQ_p(\alpha)/\bbQ_p$.
\end{introthm}

\begin{remark}
	It seems that our method for abelian and tamely ramified extensions can hardly be generalized to general extensions. For these two special cases, the key ingredient is to find an extension $K$ over $\mathbb{Q}_p(\alpha)$, which is generated by certain more ``controllable'' elements. In the abelian case, we use the cyclotomic extension by the local Kronecker-Weber theorem while in the tamely ramified case, we used the radical extension by \Cref{lem:57300new}. However, in general, we don't know how to find such a more ``controllable'' field.%
\end{remark}

\subsection{Distinguishing roots of irreducible polynomial over $\mathbb{Q}_p$}\label{subsec:48915}
The canonical expansion of an element in $\bbL_p$ is fairly an analogy of the polar coordinate of a complex number. In fact, the support $\Supp(f)$ of $f\in \bbL_p$ corresponds to the modulus of a complex number while the set $\{r_q\}_{q\in\bbQ}$ of coefficients of the expansion of $f$ corresponds to the argument of a complex number. As a result, such an expansion can be used to make a distinction of roots of polynomials over $\bbQ_p$.

Given a $p$-adic algebraic number $\alpha$, the usual arithmetic invariants (i.e. the degree, ramification index and inertia degree of the extension $\bbQ_p(\alpha)/\bbQ_p$) of $\alpha$ are determined by its minimal polynomial over $\bbQ_p$. Thus, the usual arithmetic invariants can not be used to distinguish the conjugates of $\alpha$ under the action of absolute Galois group of $\overline{\bbQ}_p$. %
We observe that in general the minimal polynomial of $\alpha$ over $\bbQ_p$ is insufficient to determine the exact value of $\frakT_\alpha$ and $\frakF_\alpha$. 
\begin{example}[cf. \Cref{eg:32730}]
	Let $\alpha_1=p^{1/p}$ and $\alpha_2=p^{1/p}\cdot \zeta_p$, where $\zeta_p\in\overline{\bbQ}_p$ is a $p$-th primitive root of unity. Then $\alpha_1$ and $\alpha_2$ share the same minimal polynomial $T^p-p$ over $\bbQ_p$, but $\frakT_{\alpha_1}=\frakF_{\alpha_1}=1$ while $\frakT_{\alpha_2}=p-1$ and $\frakF_{\alpha_2}=2$.
\end{example}
Thus, it provides the possibility to make a distinction of root of a polynomial using these two new invariants.

On the other hand, for a $p$-adic algebraic number $\alpha$, its classical arithmetic invariants are related to the hyper-algebraic invariants of all its conjugates.
The above example suggests that it makes sense to consider the hyper-algebraic invariants of all conjugate of $\alpha$ at the same time. Let $\frakT(\alpha)$ (resp. $\frakF(\alpha)$) be the set of hyper-tame indices (resp. hyper-inertia indices) of all the conjugates of $\alpha$, equipped with the partial order defined by divisibility. A small-scale numerical experiment indicates the following heuristic patterns:
\begin{enumerate}
	\item The degree of the minimal polynomial of $\alpha$ over $\mathbb{Q}_p$ is always an upper bound of $\frakF(\alpha)$ in $\bbZ_{>0}$ with respect to the order defined by divisibility.
	\item The $p$-power-free part of the ramification index of the field $\bbQ_p(\alpha)$ over $\bbQ_p$ is always the unique minimal element in $\frakT(\alpha)$.
\end{enumerate}

\subsection{Related works}
We mention some potential approaches to study the canonical expansion of general $p$-adic algebraic numbers in $\Lpha$:
\begin{enumerate}
	\item In \cite[Theorem 13.4]{kedlayaAlgebraicityGeneralizedPower2017b}, Kedlaya gives a characterization of the canonical expansion of elements of $\calO_{\bbC_p}$ in $\bbL_p$ in terms of the so-called ``$p$-quasi-automatic elements''. Extracting additional arithmetic information from these logic-derived objects could offer a fresh perspective on comprehending the hyper-algebraic invariants.
	\item In \cite{lisinskiApproximationAlgebraicityPositive2023}, Lisinski uses a variant of Newton algorithm to give an upper bound of the order type of $\Supp(\alpha)$ for element $\alpha$ in $\overline{\bbF_p(\!(t)\!)}\subset\overline{\bbF}_p(\!(t^\bbQ)\!)$. Besides that, Lisinski also designs an algorithm to give upper bounds for the characteristic $p$ analog of hyper-algebraic invariants for elements in $\overline{\bbF_p(\!(t)\!)}\subset\overline{\bbF}_p(\!(t^\bbQ)\!)$. It is possible to develop a mixed-characteristic analog of Lisinski's results for $\overline{\bbQ}_p\subset\bbL_p$ and to compare with \Cref{prop:47781} of this paper.
	\item Inspired by the pioneering work \cite{Dong2024} of Dong-He-Jin-Schremmer-Yu, which using machine learning approach to study the geometry of affine Deligne-Lusztig varieties, we wonder if the machine learning method can help to identify hidden structures in the canonical expansion of a $p$-adic algebraic number in $\Lpha$.
\end{enumerate}

\paragraph{\textbf{Acknowledgement}} The authors would like to express their gratitude to Haotian Liu, Yicheng Tao and Hongteng Xu for the useful discussion. We also thank Will Sawin for his answer on Mathoverflow (cf. \cite{463800}), which motivates the formulation of \Cref{lem:15272} in this work, and thank the anonymous referee for the careful reading and valuable suggestions that help to improve the presentation of this article. Special thanks to Beijing PARATERA Tech Corp., Ltd for providing a part of the HPC resources for the numerical experiment at the early stage of this research. The first author is supported by the Fundamental Research Funds for the Central Universities and the Research Funds of Renmin University of China \textnumero 20XNLG04, and The National Natural Science Foundation of China (Grant \textnumero 11971035).

\section{Preliminaries on field of Mal'cev-Neumann series}
In this section, we recall the definition and some basic properties of Mal'cev-Neumann fields following \cite{Poonen1993}.
\begin{definition}[{\cite[Section 3]{Poonen1993}}]
	Let $R$ be a commutative ring and $G$ be an ordered group.
	\begin{enumerate}
		\item For any $f\in\opn{Hom}_{\mathbf{Set}}(G,R)$, we define the \textbf{support} of $f$ to be $$\opn{supp}(f)=\{g\in G\colon f(g)\neq 0\}.$$
		\item Define the set of \textbf{Mal'cev-Neumann series} over $R$ with value group $G$ to be
		      $$R(\!(G)\!)\coloneqq \{f\in \opn{Hom}_{\mathbf{Set}}(G,R)\colon \opn{supp}(f) \text{ is well-ordered}\}.$$
		      By introducing a formal variable $t$, we also denote by $R(\!(t^G)\!)$ the set $R(\!(G)\!)$. The elements in $R(\!(G)\!)$ will also be written as $\sum_{g\in G}r_gt^g$, where $r_g\in R$ for all $g\in G$.
	\end{enumerate}
\end{definition}
\begin{proposition}[{\cite[Lemma 1,Corollary 2]{Poonen1993}}]
	Let $R$ be a commutative ring and $G$ be an ordered group.
	\begin{enumerate}
		\item With identity $1\cdot t^0$ and addition as well as multiplication given by
		      $$\sum_{g\in G}a_g t^g+\sum_{g\in G}b_g t^g\coloneqq\sum_{g\in G}(a_g+b_g)t^g,\ \sum_{g\in G}a_g t^g\cdot\sum_{g\in G}b_g t^g\coloneqq\sum_{g\in G}\left(\sum_{h\in G}a_h b_{g-h}\right)t^g,$$
		      $R(\!(t^G)\!)$ forms a commutative ring.
		\item If $R$ is a field, then so does $R(\!(t^G)\!)$. Moreover, with the map
		      $$v\colon R(\!(t^G)\!)\lto G\cup\{\infty\},\ f\longmapsto \begin{cases}\min\opn{supp}(f),& \text{ if }f\neq 0\\\infty,& \text{ if }f=0\end{cases},$$
		      $R(\!(t^G)\!)$ becomes a valued field with value group $G$ and residue field $R$.
	\end{enumerate}
\end{proposition}

\begin{theorem}[{\cite[Proposition 3, Corollary 3, Proposition 5]{Poonen1993}}]
	Let $k$ be a perfect field of characteristic $p$ and $G$ be an ordered group containing $\bbZ$ as a subgroup. Besides that, let
	$$\calN\coloneqq \left\{\sum_{g\in G}r_g t^g\in W(k)(\!(t^G)\!)\colon \text{ for every } g\in G,\ \sum_{n\in\bbZ}r_{g+n}p^n=0\right\},$$
	where $W(k)$ is the ring of Witt vectors of $k$. Then
	\begin{enumerate}
		\item $\calN$ is a maximal ideal of $W(k)(\!(t^G)\!)$, which makes $W(k)(\!(p^G)\!)\coloneqq W(k)(\!(t^G)\!)/\calN$ a field\footnote{Intuitively speaking, $W(k)(\!(p^G)\!)$ is obtained by replacing the formal variable $t$ of elements in $W(k)(\!(t^G)\!)$ by the prime $p$.}, called the \textbf{$p$-adic Mal'cev-Neumann field}.
		\item Every element in $W(k)(\!(p^G)\!)$ can be uniquely (and formally) written as $$\sum_{g\in G}[r_g]p^g,$$
		      where $r_g\in k$ for all $g\in G$ and $[\cdot]\colon k\lto W(k)$ is the Teichmüller lift.
		\item For $f=\sum_{g\in G}[r_g]p^g$, define the \textbf{support} of $f$ to be
		      $$\opn{supp}(f)=\{g\in G\colon r_g\neq 0\}.$$
		      Then the map
		      $$v\colon W(k)(\!(G)\!)/\calN\lto G\cup\{\infty\},\ f\mapsto \begin{cases}
				      \min\opn{supp}(f), & \text{ if }f\neq 0 \\\infty,& \text{ if }f=0
			      \end{cases}$$
		      makes $W(k)(\!(G)\!)/\calN$ a mixed-characteristic valued field with value group $G$ and residue field $k$.
	\end{enumerate}
\end{theorem}
\begin{example}
	Let $\bbL_p\coloneqq W(\overline{\bbF}_p)(\!(p^\bbQ)\!)$ be the $p$-adic Mal'cev-Neumann field with value group $\bbQ$ and residue field $\overline{\bbF}_p$. It can be shown that $\bbL_p$ is complete and algebraically closed (cf. \cite[Proposition 6, Theorem 1]{Poonen1993}). Moreover, $\bbL_p$ is the unique (up to non-canonical isomorphism) minimal spherically complete extension of $\bbC_p$ (cf. \cite[Corollary 6]{Poonen1993}), with transcendence degree $2^{\aleph_0}$ over $\bbC_p$ (cf. \cite[Corollary 9]{Poonen1993}).
\end{example}

\subsection*{Newton algorithm on $\bbL_p$}
In \cite{kedlayaPowerSeriesPAdic2001}, Kedlaya gives a constructive proof of the algebraic closeness of $\bbL_p$ by using a transfinite Newton algorithm, which we extract in the following.

For a non-constant polynomial $P(T)=\sum_{i=0}^n a_{n-i}T^i\in \bbL_p[T]$, denote by $\Newt{P}$ the Newton polygon of $P$, i.e. the lower boundary of the convex hull of the set of points $(i,v_p(a_i))$ for $i=0,1,\cdots,n$. We write $s_{\max}^P$ for the slope of the segment of $\Newt{P}$ with the largest slope and $m_{\max}^P$ the left endpoint of this segment. Besides that, call
$$\opn{Res}_P(T)\coloneqq \sum_{k=0}^{n-m_{\max}^P}C_{v_p(a_m)+s_{\max}^P(n-m_{\max}^P-k)}\left(a_{n-k}\right)T^k$$
the residue polynomial of $P$, where for any $s\in\bbQ$, the map
$C_s\colon\bbL_p\lto \overline{\bbF}_p$ is given by $\sum_{q\in\bbQ}[r_q]p^q\longmapsto r_s$.

We extracted Kedlaya's proof into the following pseudocode:
\begin{algorithm}[H]
	\caption{transfinite Newton algorithm for \(\bbL_p\)}    \begin{algorithmic}
		\INPUT A non-constant polynomial \(P(T)\in \bbL_p[T]\)
		\OUTPUT A root of \(P(T)\) in \(\bbL_p\)
		\State \(r\gets 0, \Phi(T)\gets P(T)\) \Comment{We denote the coefficient of \(T^i\) in \(\Phi\) as \(b_{n-i}\).
		}
		\While{\(\Phi(0)\neq 0\)}\Comment{This loop runs transfinitely.}
		\State \(c\gets\) any root of \(\mathrm{Res}_{\Phi}(T)\) in \(\bar{\bbF}_p\)
		\State \(r\gets r+[c]\cdot p^{s_{\max}^\Phi}\)
		\State \(\Phi(T)\gets \Phi(T+[c]\cdot p^{s_{\max}^\Phi})\)
		\EndWhile
		\State \Return \(r\)
	\end{algorithmic}
\end{algorithm}
We refer to \cite{WangYuan2021} for a full explanation of this algorithm.

Let $r=\sum_{\omega}[c_\omega]p^{k_\omega}\in\bbL_p$, with ordinal $\omega$ runs through the well-ordered set $\Supp(r)$, be a root of $P(T)$ given by the above algorithm. For the convenience of later discussion, we call $r_\omega=\sum_{r<\omega}[c_\omega]p^{k_\omega}$ the $\omega$-th approximation of $r$, $P_\omega=P(T+r_\omega)$ the $\omega$-th approximation polynomial
and $\opn{Res}_{P_\omega}(T)$ the $\omega$-th residue polynomial.

\begin{example}[\cite{WangYuan2021,OurArxiv}]\label[example]{eg:29565}
	Let $\zeta_{2(p-1)}\in W(\bbF_{p^2})$ be a $2(p-1)$-th primitive root of unity.
	\begin{enumerate}
		\item There exist a $p$-th root of unity, whose canonical expansion in $\bbL_p$ is given by
		      $$\zeta_p=\sum_{k=0}^{p-1}\frac{\zeta_{2(p-1)}^k}{k!}p^{\frac{k}{p-1}}+\sum_{k=p}^\infty [c_k]p^{\frac{k}{p-1}},$$
		      where $c_k\in\bbF_{p^2}$ for all $k\geq p$. By changing the choice of $\zeta_{2(p-1)}$, every $p$-th primitive root of unity can be expanded in the above form.
		\item For $n\geq 2$, there exists a $p^n$-th root of unity, whose (non-canonical) expansion in $\bbL_p$ is partially given by
		      \begin{align*}
			      \zeta_{p^n}= & \sum_{k=0}^{p-1}\frac{(-1)^{nk}}{k!}\zeta_{2(p-1)}^k p^{\frac{k}{p^{n-1}(p-1)}}+\sum_{k=0}^{p-1}\frac{(-1)^{n(k+1)}}{k!}\zeta_{2(p-1)}^{k+1}p^{\frac{k+p}{p^{n-1}(p-1)}}\left(\sum_{l=n}^\infty p^{-1/p^l}\right) \\
			                   & \quad-\sum_{k=1}^{p-1}\frac{(-1)^{n(k+1)}}{k!}\left(\sum_{l=1}^k \frac{1}{l}\right)\zeta_{2(p-1)}^{k+1}p^{\frac{k+p}{p^{n-1}(p-1)}}                                                                               \\
			                   & \quad+\frac{1}{2}\zeta_{2(p-1)}^2p^{\frac{2}{p^{n-2}(p-1)}}\left(\sum_{l=n}^\infty p^{-1/p^l}\right)^2+\frac{(-1)^n}{2}\zeta_{2(p-1)}^3 p^{\frac{2}{p^{n-2}(p-1)}-\frac{p-2}{p^n(p-1)}}                           \\
			                   & \quad +\text{ terms with higher valuation}\cdots.
		      \end{align*}
	\end{enumerate}
\end{example}

\section{Field of hyper-algebraic elements in $\bbL_p$}
\subsection{Hyper-algebraic elements}

\begin{definition}
	We call an element $f=\sum_{q\in \bbQ}[r_q]p^q\in\bbL_p$ \textbf{hyper-algebraic}, if it satisfies:
	\begin{enumerate}
		\item there exists a positive integer $N$ such that $\opn{supp}(f)\subseteq \frac{1}{N}\bbZ[1/p]$;
		\item there exists a positive integer $k$ such that $r_q\in\bbF_{p^k}$ for all $q\in\opn{supp}(f)$.
	\end{enumerate}
	Denote by $\Lpha$ the set of all hyper-algebraic elements in $\bbL_p$.
\end{definition}
By \cite[Corollary 8]{Poonen1993}, we know that
\begin{proposition}
	The set $\Lpha$ forms an algebraically closed field. As a consequence, all $p$-adic algebraic numbers are hyper-algebraic, i.e. $\overline{\bbQ}_p\subseteq \Lpha$.
\end{proposition}
We clarify the relations among the fields $\Lpha$, $\overline{\bbQ}_p$ and $\bbC_p$:
\begin{theorem}\label{prop:54163}\leavevmode
	\begin{enumerate}
		\item The fields $\Lpha$ and  $\bbC_p$ do not contain each other. In particular, $\Lpha$ contains $\overline{\bbQ}_p$ as a proper subfield.
		\item The field $\Lpha$ is not complete, and its completion is a proper subfield of $\bbL_p$.
	\end{enumerate}
\end{theorem}
\begin{proof}
	Consider the following element of $\Lpha$:
	$$\alpha=\sum_{k=1}^{\infty}p^{\frac{\lfloor\sqrt{2}\cdot p^k\rfloor}{p^k}}.$$
	If $\alpha\in\bbC_p$, then there exists a $p$-adic algebraic number $\beta\in\overline{\bbQ}_p$ that $v_p(\alpha-\beta)>2$. This shows that the canonical expansion of $\beta$ in $\Lpha$ has the form
	$$\beta=\sum_{k=1}^{\infty}p^{\frac{\lfloor\sqrt{2}\cdot p^k\rfloor}{p^k}}+\text{ terms with exponent greater than $2$}.$$
	Thus, $\Supp(\beta)$ has accumulation value $\sqrt{2}$. However, this is impossible:	Lampert showed in \cite[Theorem 2]{lampertAlgebraicPadicExpansions1986} that the set
	$$\calA\coloneqq\left\{\alpha\in\bbL_p\middle\vert \{\text{accumulation value of }\Supp(\alpha)\}\subset \bbQ\right\}$$
	is an algebraically closed field. Since the support of every $p$-adic rational number lies in $\bbZ\subset\bbQ$, $\overline{\bbQ}_p$ is a subfield of $\calA$. On the other hand, $\beta$ does not belong to $\calA$. This contradiction shows that $\Lpha$ is not contained in $\bbC_p$. In particular, $\Lpha$ contains $\overline{\bbQ}_p$ as a proper subfield.

	To show that $\Lpha$ is not complete and does not contain $\bbC_p$, we can consider
	the sequence $\left(\sum_{k=1}^n p^{k-1/k}\right)_{n\geq 1}$ in $\overline{\bbQ}_p\subseteq \Lpha$, which clearly converges in $\bbC_p$ but has non-hyper-algebraic limit $\sum_{k=1}^\infty p^{k-1/k}$ in $\bbL_p$: the $p$-power-free part of the denominators of elements of its support is unbounded.%

	To prove $\Lpha$ is not dense in $\bbL_p$, we consider the element $\gamma=\sum_{k=1}^\infty p^{-\frac{1}{p^k-1}}$ in $\bbL_p$. If it lies in the completion of $\Lpha$, then there exists an element $\delta\in \Lpha$ that $v_p(\gamma-\delta)>1$. This leads to a contradiction if we consider the canonical expansion of $\delta$ in $\Lpha$
	$$\delta=\sum_{k=1}^\infty p^{-\frac{1}{p^k-1}}+\text{ terms with exponent greater than 1}.$$
	The $p$-power-free part of the denominators of elements of $\Supp(\delta)$ are unbounded, suggesting that $\delta$ is not hyper-algebraic.
\end{proof}

\subsection{Hyper-tame index and hyper-inertia index}
\begin{definition}
	Let $\theta=\sum_{q\in \bbQ}[r_q]p^q\in\Lpha$ be a hyper-algebraic element in $\bbL_p$.
	\begin{enumerate}
		\item Denote by $\frakT_\theta$ the minimal positive integer $e$ such that $\opn{supp}(\theta)\subseteq \frac{1}{e}\bbZ[1/p]$. We call it the \textbf{hyper-tame index} of $\theta$.
		\item Denote by $\frakF_\theta$ the minimal positive integer $f$ such that $r_q\in\bbF_{p^f}$ for all $q\in\opn{supp}(\theta)$. We call it the \textbf{hyper-inertia index} of $\theta$.
	\end{enumerate}
	We call them the \textbf{hyper-algebraic invariants} of $\theta$.
\end{definition}

The following lemmas collect several basic properties of the hyper-tame and hyper-inertia indices:
\begin{lemma}\label[lemma]{lem:40178}
	Let $\alpha=\sum_{q\in\bbQ}[r_q]p^q$ be a hyper-algebraic element in $\bbL_p$. Then one has
	\begin{enumerate}
		\item the hyper-tame index $\frakT_\alpha$ is coprime to $p$;
		\item If the set of coefficients $\{r_q\}_{q\in\bbQ}$ is contained in a finite field $\bbF_{p^s}$, then $s$ is a multiple of $\frakF_\alpha$;
		\item If the support $\Supp(\alpha)$ is contained in the set $\frac{1}{N}\bbZ[1/p]$ for some positive integer $N$, then $N$ is a multiple of $\frakT_\alpha$;
	\end{enumerate}
\end{lemma}
\begin{proof}\leavevmode
	\begin{enumerate}[wide]
		\item For any positive integer $N$, the sets $\frac{1}{pN}\bbZ[1/p]$ and $\frac{1}{N}\bbZ[1/p]$ are identical.
		\item One has
		      $$\{r_q\}_{q\in\bbQ}\subseteq \bbF_{p^{\frakF_\alpha}}\cap\bbF_{p^s}=\bbF_{p^{\gcd(\frakF_\alpha,s)}}.$$
		      The result follows from the minimality of $\frakF_\alpha$.
		\item By the first assertion, we may assume that $N$ is coprime to $p$. Suppose the contrary that $N=d\cdot\frakT_\alpha+r$ with $d\in\bbZ_{\geq 1}$ and $r\in\{1,\cdots,\frakT_\alpha-1\}$. Take $q\in\Supp(\alpha)$. Then the inclusion $q\in\frac{1}{\frakT_\alpha}\bbZ[1/p]\cap\frac{1}{N}\bbZ[1/p]$ allows us to write $$q=\frac{a_1\cdot p^{v_1}}{\frakT_\alpha}=\frac{a_2\cdot p^{v_2}}{N},$$
		      where $a_1,a_2,v_1,v_2\in\bbZ$ with $a_1,a_2$ coprime to $p$. By comparing the $p$-adic valuation, we get $v_1=v_2$. Since
		      $$a_2\cdot p^{v_2}=(d\cdot \frakT_\alpha+r)\cdot q=d\cdot a_1\cdot p^{v_1}+r\cdot q=d\cdot a_1\cdot p^{v_2}+r\cdot q,$$
		      we obtain that $q=\frac{a_2-d\cdot a_1}{r}\cdot p^{v_2}\in\frac{1}{r}\bbZ[1/p]$, which contradicts the minimality of $\frakT_\alpha$.
	\end{enumerate}
\end{proof}

\begin{lemma}\label{lem:36325}
	Let $\alpha,\beta\in\Lpha$ be two hyper-algebraic elements in $\bbL_p$. Then one has
	\begin{enumerate}
		\item $\frakT_{\alpha+\beta}\mid\opn{lcm}(\frakT_\alpha,\frakT_\beta)$, $\frakF_{\alpha+\beta}\mid \opn{lcm}(\frakF_\alpha,\frakF_\beta)$.
		\item $\frakT_{\alpha\cdot\beta}\mid\opn{lcm}(\frakT_\alpha,\frakT_\beta)$, $\frakF_{\alpha\cdot\beta}\mid \opn{lcm}(\frakF_\alpha,\frakF_\beta)$. In particular if $\alpha$ is algebraic over $\bbQ_p$ and $\bbQ_p(\alpha)$ is unramified over $\bbQ_p$, then $\frakT_{\alpha\cdot\beta}\mid \frakT_\beta$ and $\frakF_{\alpha\cdot \beta}\mid \opn{lcm}(\frakf_\alpha,\frakF_\beta)$.
		\item $\frakT_{1/\alpha}=\frakT_\alpha$, $\frakF_{1/\alpha}=\frakF_\alpha$ for $\alpha\neq 0$.
	\end{enumerate}
\end{lemma}
\begin{proof}
	The first and the second assertions follow from the definition of addition and multiplication on $\bbL_p$. In particular if $\bbQ_p(\alpha)$ is unramified over $\bbQ_p$, then $\bbQ_p(\alpha)=\opn{Frac}W(\bbF_{p^{\frakf_\alpha}})$. As a result, every element in $\bbQ_p(\alpha)$ has the form $\sum_{k>\!\!\!>-\infty}[\zeta_k]p^k$, where $\zeta_k\in\bbF_{p^{\frakf_\alpha}}$ for all $k$. This shows that $\frakT_\alpha=1$ and $\frakF_\alpha=\frakf_\alpha$.

	For the third assertion, the result is trivial when $\vert\Supp(\alpha)\vert=1$. Now we suppose $\vert\Supp(\alpha)\vert\geq 2$ and write $\alpha=[\zeta]p^{v_p(\alpha)}-A$ for some $\zeta\in\overline{\bbF}_p$ with $v_p(A)>v_p(\alpha)$. Then $\zeta\in\bbF_{p^{\frakF_\alpha}}$, $\frakT_A\mid\frakT_\alpha$ and $\frakF_A\mid \frakF_\alpha$. The result follows from the expansion
	$$\alpha^{-1}=[\zeta^{-1}]p^{-v_p(\alpha)}\sum_{k=0}^\infty \left([\zeta^{-1}]p^{-v_p(\alpha)}\cdot A\right)^k,$$
	where $v_p\left([\zeta^{-1}]p^{-v_p(\alpha)}\cdot A\right)>0,\ \frakT_{[\zeta^{-1}]p^{-v_p(\alpha)}\cdot A}\mid\frakT_\alpha \text{ and }\frakF_{[\zeta^{-1}]p^{-v_p(\alpha)}\cdot A}\mid \frakF_\alpha$.
\end{proof}

\begin{corollary}\label[corollary]{coro:9860}
	For any positive integer $e,f\geq 1$, the set
	$$\Lpha(e,f)\coloneqq \{\alpha\in\Lpha\colon \frakF_\alpha\mid f,\ \frakT_\alpha\mid e\}$$
	is a subfield of $\Lpha$. In particular, for any $\alpha\in\Lpha$, we have $\bbQ_p(\alpha)\subset\Lpha(\frakT_\alpha,\frakF_\alpha)$.
\end{corollary}
\section{$p$-adic algebraic numbers in $\Lpha$}
The objective of this section is to investigate the hyper-algebraic invariants of $p$-adic algebraic numbers.%
\subsection{ Hyper-algebraic invariants of general $p$-adic algebraic numbers}

As observed in \cite[Corollary 8]{Poonen1993}, there are two special types of automorphisms in $\opn{Aut}_{\bbQ_p}(\bbL_p)$:
\begin{enumerate}
	\item for any $g\in\calG_{\bbF_p}\coloneqq\opn{Gal}\left(\overline{\bbF}_p/\bbF_p\right)$, it can be viewed as an element of $\opn{Aut}_{\bbQ_p}(\bbL_p)$ by the formula
	      $$g\cdot\sum_{q\in\bbQ}[r_q]p^q=\sum_{q\in\bbQ}[g(r_q)]p^q,\ \text{for any }\sum_{q\in\bbQ}[r_q]p^q\in\bbL_p.$$
	Since $\bbQ_p\subsetneq\bbL_p$ is fixed by this action of $\calG_{\bbF_p}$, it indeed gives rise to a group homomorphism $\calG_{\bbF_p}\lto\calG_{\bbQ_p}$, which is a section of the natural morphism $\calG_{\bbQ_p}\lto\calG_{\bbF_p}$.

	\item For any group homomorphism $\xi\colon\bbQ/\bbZ\lto \overline{\bbF}_p^\times$, the following formula
	      $$\lambda_\xi\colon \sum_{q\in\bbQ}[r_q]p^q\longmapsto \sum_{q\in\bbQ}[\xi(q)r_q]p^q,\ \text{for any }\sum_{q\in\bbQ}[r_q]p^q\in\bbL_p$$
	      also gives an element of $\opn{Aut}_{\bbQ_p}(\bbL_p)$.
\end{enumerate}
These two types of automorphisms can be used to show that the hyper-algebraic invariants of $p$-adic algebraic numbers are bounded from above by their degrees over $\bbQ_p$. 

We first prove this assertion for the hyper-inertia index.
\begin{proposition}\label[proposition]{prop:47781}
	For every $p$-adic algebraic number $\alpha$, one has $\frakF_\alpha\leq[\bbQ_p(\alpha)\colon\bbQ_p]$.
\end{proposition}
\begin{proof}

	As we have mentioned, the action of $\calG_{{\bbF}_p}$ on $\bbL_p$ sends $p$-adic algebraic numbers to their conjugates under the action of $\calG_{\bbQ_p}$. As a result, one has
	\begin{equation}\label{eq:47781}
		\left\lvert\{g(\alpha)\colon g\in\calG_{\bbF_p}\}\right\rvert\leq [\bbQ_p(\alpha)\colon\bbQ_p]
	\end{equation}
	for any $\alpha\in\overline{\bbQ}_p$.

	Suppose there exists $\alpha\in\overline{\bbQ}_p$ that $\frakF_\alpha>[\bbQ_p(\alpha)\colon\bbQ_p]$. Thus, there exists a rational number $q_0\in\Supp(\alpha)$ such that the minimal finite field containing $C_{q_0}(\alpha)$ is $\bbF_{p^r}$ with $r>[\bbQ_p(\alpha)\colon\bbQ_p]$. Consider the set
	$$\{g(C_{q_0}(\alpha))\colon g\in\calG_{\bbF_p}\}=\left\{C_{q_0}(\alpha),C_{q_0}(\alpha)^p,\cdots,C_{q_0}(\alpha)^{p^n},\cdots\right\}.$$
	The cardinality of this set is the minimal positive integer $d$ that $C_{q_0}(\alpha)=C_{q_0}(\alpha)^{p^d}$, which is the same as the minimal positive integer $d$ that $C_{q_0}(\alpha)\in\bbF_{p^d}$. This shows that $r=d$, i.e. $\left\lvert\{g(C_{q_0}(\alpha))\colon g\in\calG_{\bbF_p}\}\right\rvert=r$. Since the following map is surjective
	$$\{g(\alpha)\colon g\in\calG_{\bbF_p}\}\lto \{g(C_{q_0}(\alpha))\colon g\in\calG_{\bbF_p}\},\ g(\alpha)\longmapsto C_{q_0}(g(\alpha))=g(C_{q_0}(\alpha)),$$
	we know that $\left\lvert\{g(\alpha)\colon g\in\calG_{\bbF_p}\}\right\rvert\geq r$, which contradicts to \eqref{eq:47781}.
\end{proof}
In the rest of this subsection, we are going to show that the hyper-tame index of a $p$-adic algebraic number is also bounded from above by its degree over $\bbQ_p$. To this end, we need to introduce the notion of admissible maps. Denote by $\mathbf{Set}$ (resp. $\mathbf{Ab}$) the category of sets (resp. abelian groups).
\begin{definition}%
	Let $M$ be a subset of $\bbQ/\bbZ$. A map $f\colon M\lto \overline{\bbF}_p^\times$ is called admissible, if it can be extended to a (non necessarily unique) group homomorphism $\widetilde{f}\in\opn{Hom}_{\mathbf{Ab}}\left(\bbQ/\bbZ,\overline{\bbF}_p^\times\right)$.
\end{definition}
For any $\alpha\in \bbL_p$, we denote by $\opn{Hom}_{\mathbf{Set}}^{\opn{adm}}\left(\Supp(\alpha)/\bbZ,\overline{\bbF}_p^\times\right)$ the set of all admissible maps from $\Supp(\alpha)/\bbZ$ to $\overline{\bbF}_p^{\times}$.
If $\alpha=\sum_{q\in\bbQ}[r_q]p^q\in\overline{\bbQ}_p$ and $f\in \opn{Hom}_{\mathbf{Set}}^{\opn{adm}}\left(\Supp(\alpha)/\bbZ,\overline{\bbF}_p^\times\right)$, then we have
$$\sum_{q\in\bbQ}[f(q)r_q]p^q=\sum_{q\in\bbQ}[\widetilde{f}(q)r_q]p^q \in \overline{\bbQ}_p$$
for any extension $\widetilde{f}\in\opn{Hom}_{\mathbf{Ab}}\left(\bbQ/\bbZ,\overline{\bbF}_p^\times\right)$. Note that for any group homomorphism $\xi\colon\bbQ/\bbZ\lto \overline{\bbF}_p^\times$, $\lambda_\xi$ maps $p$-adic algebraic numbers to their conjugates under the action of $\calG_{\bbQ_p}\coloneqq \opn{Gal}\left(\overline{\bbQ}_p/\bbQ_p\right)$. This gives us an injective map:
$$\Phi_\alpha\colon\opn{Hom}^{\opn{adm}}_{\mathbf{Set}}\left(\Supp(\alpha)/\bbZ,\overline{\bbF}_p^\times\right)\lto\{g(\alpha)\colon g\in\calG_{\bbQ_p}\},\ f\longmapsto \lambda_{\widetilde{f}}(\alpha).$$

\begin{lemma}\label{lem:61266}
	Let $A$ be a subset of $\bbQ$ and let $\langle A/\bbZ\rangle$ be the subgroup of $\bbQ/\bbZ$ generated by $A/\bbZ$. Then
	\begin{lemenum}
		\item\label{it:57329} One has a bijection:
		$$\opn{Hom}_{\mathbf{Ab}}\left(\langle A/\bbZ\rangle,\overline{\bbF}_p^\times\right)\lto\opn{Hom}^{\opn{adm}}_{\mathbf{Set}}\left(A/\bbZ,\overline{\bbF}_p^\times\right).$$
		\item\label{it:41238} If $\opn{Hom}^{\opn{adm}}_{\mathbf{Set}}\left(A/\bbZ,\overline{\bbF}_p^\times\right)$ is a finite set, then
		$A\subseteq \frac{1}{N}\bbZ[1/p]$,
		where $N=\left\lvert\opn{Hom}_{\mathbf{Ab}}\left(\langle A/\bbZ\rangle,\overline{\bbF}_p^\times\right)\right\rvert$.
	\end{lemenum}
\end{lemma}
\begin{proof}\leavevmode
	\begin{enumerate}[wide]
		\item By restricting the morphisms in $\opn{Hom}_{\mathbf{Ab}}\left(\langle A/\bbZ\rangle,\overline{\bbF}_p^\times\right)$ to $A/\bbZ$, we obtain an injection
		      $$\iota\colon\opn{Hom}_{\mathbf{Ab}}\left(\langle A/\bbZ\rangle,\overline{\bbF}_p^\times\right)\lto\opn{Hom}_{\mathbf{Set}}\left(A/\bbZ,\overline{\bbF}_p^\times\right).$$
		      We are left to show that the image of this map is exactly $\opn{Hom}^{\opn{adm}}_{\mathbf{Set}}\left(A/\bbZ,\overline{\bbF}_p^\times\right)$.

		      For any $f\in \opn{Hom}^{\opn{adm}}_{\mathbf{Set}}\left(A/\bbZ,\overline{\bbF}_p^\times\right)$, any extension $\widetilde{f}\in\opn{Hom}_{\mathbf{Ab}}\left(\bbQ/\bbZ,\overline{\bbF}_p^\times\right)$ of $f$ has image $f$ by the injection $\iota$. This implies that $\opn{Hom}^{\opn{adm}}_{\mathbf{Set}}\left(A/\bbZ,\overline{\bbF}_p^\times\right)$ is contained in the image of $\iota$.

		      For any $h=\iota(a)\in \opn{Hom}_{\mathbf{Set}}\left(A/\bbZ,\overline{\bbF}_p^\times\right)$ with some $a\in \opn{Hom}_{\mathbf{Ab}}\left(\langle A/\bbZ\rangle,\overline{\bbF}_p^\times\right)$, $a$ extends uniquely to a group homomorphism $\widetilde{a}\in \opn{Hom}_{\mathbf{Ab}}\left(\bbQ/\bbZ,\overline{\bbF}_p^\times\right)$. This is because $\overline{\bbF}_p^\times$ is divisible, and consequently injective in $\mathbf{Ab}$. Since $\widetilde{a}\mid_{A/\bbZ}=a\mid_{A/\bbZ}=\iota(a)=h$, we know that $h$ is admissible.
		\item The following proof is given by Lahtonen (cf. \cite{4891731}). Let $N=\vert \opn{Hom}_{\mathbf{Ab}}\left(\langle A/\bbZ\rangle,\overline{\bbF}_p^\times\right)\vert$. Suppose there exists a rational number $q\in\bbQ$ that $q+\bbZ\in\langle A/\bbZ\rangle$ and $q\notin \frac{1}{N}\bbZ[1/p]$. We write $q=\frac{u}{p^r\cdot v}$, where $u,v\in\bbZ_{\geq 1}$, $r\in\bbZ_{\geq 0}$ with $\gcd(u,v)=\gcd(p,v)=1$. Since $q\notin \frac{1}{N}\bbZ[1/p]$, one knows that $v$ does not divide $N$.

		      Notice that the element $z^\prime\coloneqq \frac{u}{v}+\bbZ$ has order $v$ in $\langle A/\bbZ\rangle\subseteq \bbQ/\bbZ$. Fix a $v$-th primitive root $\zeta_v$ of unity in $\overline{\bbF}_p^\times$, then the map $z^\prime\longmapsto \zeta_v$ induces a morphism $d$ in $\opn{Hom}_{\mathbf{Ab}}\left(\langle z^\prime\rangle,\overline{\bbF}_p^\times\right)$ with order $v$. Since $\overline{\bbF}_p^\times$ is injective in $\mathbf{Ab}$, $d$ extends to a morphism $\widetilde{d}\in\opn{Hom}_{\mathbf{Ab}}\left(\langle A/\bbZ\rangle,\overline{\bbF}_p^\times\right)$. The order of $\widetilde{d}$ in $\opn{Hom}_{\mathbf{Ab}}\left(\langle A/\bbZ\rangle,\overline{\bbF}_p^\times\right)$, which divides $N$ by Lagrange's theorem, is a multiple of $v$. This contradicts the assertion that $v$ does not divide $N$. Thus, $\langle A/\bbZ\rangle\subset \frac{1}{N}\bbZ[1/p]/\bbZ$, which allows us to conclude the proof.
	\end{enumerate}

\end{proof}
\begin{proposition}\label[proposition]{prop:57726}
	For every $p$-adic algebraic number $\alpha$, one has $\frakT_\alpha\leq[\bbQ_p(\alpha)\colon\bbQ_p]$.
\end{proposition}
\begin{proof}

	We can set $A$ in \Cref{it:41238} to be $\Supp(\alpha)$, and we obtain $\Supp(\alpha)\subseteq\frac{1}{N}\bbZ[1/p]$,
	where
	$$N=\left\lvert\opn{Hom}_{\mathbf{Ab}}\left(\langle \Supp(\alpha)/\bbZ\rangle,\overline{\bbF}_p^\times\right)\right\rvert.$$
	By \Cref{it:57329}, we have $N=\left\lvert\opn{Hom}^{\opn{adm}}_{\mathbf{Set}}\left(\Supp(\alpha)/\bbZ,\overline{\bbF}_p^\times\right)\right\rvert$.
	Thus, $\frakT_\alpha\leq N\leq [\bbQ_p(\alpha)\colon\bbQ_p]$, as promised.
\end{proof}

Unlike the usual inertia degree and tamely ramification index, $\frakT_\alpha$ (resp. $\frakF_\alpha$) does not always divide $[\bbQ_p(\alpha)\colon\bbQ_p]$ for a general $p$-adic algebraic number $\alpha$. Moreover, these hyper-algebraic invariants are not stable under the action of $\calG_{\mathbb{Q}_p}$. This is illustrated by the following example:
\begin{example}\label[example]{eg:32730}
	Let $\alpha_1=p^{1/p}$ and $\alpha_2=p^{1/p}\cdot \zeta_p$, where $\zeta_p\in\overline{\bbQ}_p$ is a $p$-th primitive root of unity. Then $\alpha_1$ and $\alpha_2$ are conjugate under the action of $\calG_{\bbQ_p}$, but $\frakT_{\alpha_1}=\frakF_{\alpha_1}=1$, while $\frakT_{\alpha_2}=p-1$, $\frakF_{\alpha_2}=2$. In particular, neither $\frakT_{\alpha_2}$ nor $\frakF_{\alpha_2}$ divides $[\bbQ_p(\alpha_2)\colon\bbQ_p]=p$.

	In fact, the only non-trivial assertion in this example is the calculation of the hyper-algebraic invariants of $\alpha_2$. By \Cref{eg:29565}, $\zeta_p$ can be expanded as
	$$\zeta_p=\sum_{k=0}^\infty[c_k]\cdot p^{\frac{k}{p-1}},$$
	with $c_k\in\bbF_{p^2}$ for all $k\in\bbN$ and $c_1\notin\bbF_p$. As a result, $\alpha_2$ can be expanded as
	$$\alpha_2=\sum_{k=0}^\infty[c_k]\cdot p^{\frac{k}{p-1}+\frac{1}{p}},$$
	and the result follows immediately.
\end{example}

\subsection{Hyper-algebraic invariants of abelian extensions}
Let $\zeta_{p^n}$ be the $p^n$-th root of unity in \Cref{eg:29565}. It is easy to see that
\begin{table}[H]
	\begin{tabular}{ccc}
		\toprule
		                & $\alpha=\zeta_p$ & $\alpha=\zeta_{p^n}\ (n\geq 2)$ \\
		\midrule
		$\frakF_\alpha$ & $2$              & $\geq 2$                        \\ 
		$\frakT_\alpha$ & $p-1$            & $\geq p-1$                      \\
		\bottomrule
	\end{tabular} .
\end{table}
The following proposition gives a precise form of the above observations:
\begin{proposition}\label[proposition]{prop:14697}
	For any integer $n\geq 1$ and any $p^n$-th primitive root of unity $\zeta_{p^n}$, we have $\frakT_{\zeta_{p^n}}=p-1$ and
	$$\frakF_{\zeta_{p^n}}\begin{cases}\begin{alignedat}{3}
				 & \omit\hfill $=$ \hfill &  & 2,              &  & \text{ if }n=1,2;   \\
				 & \text{divides }        &  & 2\cdot p^{n-2}, &  & \text{ if }n\geq 3.
			\end{alignedat}\end{cases}$$
\end{proposition}
The key to prove this proposition is the following lemma:
\begin{lemma}\label[lemma]{lem:35078}
	Let $\alpha\in\Lpha$ with $v_p(\alpha)=0$. Then there exists a $p$-th root $\beta$ of $\alpha$ in $\Lpha(\frakT_\alpha,p\cdot\frakF_\alpha)$. In particular, if $C_{\frac{1}{p-1}}(\beta)=0$, then $\beta$ belongs to $\Lpha(\frakT_\alpha,\frakF_\alpha)$.
\end{lemma}
\begin{proof}
	We apply the transfinite Newton algorithm on the equation $T^p-\alpha=0$ to get a root $\beta$. Set $\beta=\sum_{\omega}[c_\omega]\cdot p^{k_\omega}$, where the ordinal $\omega$ runs through the well-ordered set $\Supp(\beta)$. Recall that for any ordinal $\omega$, let $\beta_\omega=\sum_{\rho<\omega}[c_\rho]\cdot p^{k_\rho}$ and
	$$\Phi_\omega(T)=(T+\beta_\omega)^p-\alpha=T^p+\sum_{k=1}^{p-1}\binom{p}{k}\beta_\omega^k\cdot T^{p-k}+\beta_\omega^p-\alpha.$$

	The first step is easy: since $\beta_0=0$ and $\Phi_0(T)=T^p-\alpha$, the Newton polygon $\Newt{\Phi_0}$ consists of a single horizontal segment with residue polynomial given by $$\opn{Res}_{\Phi_0}(T)=T^p-C_0(\alpha)\in\bbF_{p^{\frakF_\alpha}}[T],$$
	which splits in $\bbF_{p^{\frakF_\alpha}}$. This shows that $\beta_1\in\Lpha(\frakT_\alpha,\frakF_\alpha)$ and $v_p(\beta_1)=0$.

	For any $\omega\geq 1$, since $v_p(\beta_\omega)=v_p(\beta_1)=0$, we know that $v_p(\binom{p}{k}\beta_\omega^k)=1$ for all $k=1,2,\cdots,p-1$. This implies that $\Newt{\Phi_\omega}$ is determined by the point $(p,v_p(\beta_\omega^p-\alpha))$ for every $\omega\geq 1$.

	Since $k_\omega\in\bbQ$ increases monotonically with respect to the ordinal $\omega$, we set $\omega_0$ to be the minimal ordinal $\rho$ that satisfies $k_\rho\geq \frac{1}{p-1}$.
	\begin{enumerate}
		\item Suppose $\omega<\omega_0$ and $\beta_\rho\in\Lpha(\frakT_\alpha,\frakF_\alpha)$ for every $\rho\leq\omega$. Then $\Newt{\Phi_\omega}$ consists of a single segment with slope $k_\omega=s_{\max}^{\Phi_\omega}=\frac{1}{p}v_p(\beta_\omega^p-\alpha)<\frac{1}{p-1}$.
		      \begin{figure}[H]
			      \centering
			      \begin{tikzpicture}[scale=0.8]
				      \draw[->] (0,0) -- (5.5,0) node[right] {};
				      \draw[->] (0,0) -- (0,5.5) node[above] {};

				      \draw (1,0pt) -- (1,-2pt) node[anchor=north] {$1$};
				      \draw (2,0pt) -- (2,-2pt) node[anchor=north] {$2$};
				      \draw (4,0pt) -- (4,-2pt) node[anchor=north] {$p-1$};
				      \draw (5,0pt) -- (5,-2pt) node[anchor=north] {$p$};
				      \draw (0pt,2) -- (-2pt,2) node[anchor=east] {$1$};
				      \fill[gray!70] (1,2) circle (0.05);
				      \fill[gray!70] (2,2) circle (0.05);
				      \fill[gray!70] (4,2) circle (0.05);
				      \fill (5,1.5) circle (0.05);
				      \node at (5,1.5) [above right]{$v_p(\beta_\omega^p-\alpha)$};
				      \draw[line width=0.7pt] (0,0) -- (5,1.5);
				      \draw[dotted] (0,2) -- (4,2);
				      \draw[dotted] (4,0) -- (4,2);
				      \draw[dotted] (5,1.5) -- (5,0);
			      \end{tikzpicture}
			      \caption{$\Newt{\Phi_{\omega}}$, $1\leq\omega<\omega_0$}
		      \end{figure}
		      Since $\beta_\omega^p-\alpha\in\Lpha(\frakT_\alpha,\frakF_\alpha)$ by \Cref{coro:9860}, we know that
		      \[v_p(\beta_\omega^p-\alpha)\in \Supp(\beta_\omega^p-\alpha)\subseteq \frac{1}{\frakT_\alpha}\bbZ[1/p].\] This implies that $k_\omega=\frac{1}{p}v_p(\beta_\omega^p-\alpha)$ also belongs to $\frac{1}{\frakT_\alpha}\bbZ[1/p]$.
		      The residue polynomial of $\Phi_\omega(T)$ is given by
		      $$\opn{Res}_{\Phi_\omega}(T)=T^p+C_{v_p(\beta_\omega^p-\alpha)}(\beta_\omega^p-\alpha),$$
		      where $C_{v_p(\beta_\omega^p-\alpha)}(\beta_\omega^p-\alpha)\in\bbF_{p^{\frakF_\alpha}}$. Thus, any root of this residue polynomial lies in $\bbF_{p^{\frakF_\alpha}}$. This shows that $\beta_{\omega+1}\in\Lpha(\frakT_\alpha,\frakF_\alpha)$. Since the case of limit ordinals is self-indicating, we can show by transfinite induction that  $\beta_{\omega}\in\Lpha(\frakT_\alpha,\frakF_\alpha)$ for all $\omega\leq \omega_0$.
		\item Now we deal with $\omega=\omega_0+1$.
		      \begin{enumerate}
			      \item If $k_{\omega_0}=s_{\max}^{\Phi_{\omega_0}}=\frac{1}{p-1}$, then $\Newt{\Phi_{\omega_0}}$ consists of a single segment with slope equals to
			            $$k_{\omega_0}=\frac{1}{p-1}=\frac{1}{p}v_p(\beta_{\omega_0}^p-\alpha)\in \frac{1}{\frakT_\alpha}\bbZ[1/p].$$
			            Since this segment contains the point $(p-1,1)$, one knows that
			            $$\opn{Res}_{\Phi_{\omega_0}}(T)=T^p+C_0(\beta_{\omega_0})^{p-1}T+C_{v_p(\beta_{\omega_0}^p-\alpha)}(\beta_{\omega_0}^p-\alpha)\in\bbF_{p^{\frakT_\alpha}}[T],$$
			            whose root lies in $\bbF_{p^{p\cdot \frakF_\alpha}}$. In this case, one has $\beta_{\omega_0+1}\in \Lpha(\frakT_{\alpha},p\cdot\frakF_\alpha)$.
			      \item If $k_{\omega_0}=s_{\max}^{\Phi_{\omega_0}}>\frac{1}{p-1}$, then $\Newt{\Phi_{\omega_0}}$ consists of two segments, where the vertexes of the segment with maximal slope is given by $(p-1,1)$ and $(p,v_p(\beta_{\omega_0}^p-\alpha))$. Thus,
			            $$k_{\omega_0}=\frac{v_p(\beta_{\omega_0}^p-\alpha)-1}{p-(p-1)}\in \frac{1}{\frakT_\alpha}\bbZ[1/p]$$
			            and one has
			            $$\opn{Res}_{\Phi_{\omega_0}}(T)=C_0(\beta_{\omega_0})^{p-1}T+C_{v_p(\beta_{\omega_0}^p-\alpha)}(\beta_{\omega_0}^p-\alpha),$$
			            whose root lies in $\bbF_{p^{\frakF_\alpha}}$. In this case, one has $\beta_{\omega_0+1}\in \Lpha(\frakT_{\alpha},\frakF_\alpha)$.
		      \end{enumerate}
		      \begin{figure}[H]
			      \captionsetup{justification=centering}
			      \centering
			      \begin{minipage}[t]{0.49\textwidth}
				      \centering
				      \begin{tikzpicture}[scale=0.8]
					      \draw[->] (0,0) -- (5.5,0) node[right] {};
					      \draw[->] (0,0) -- (0,5.5) node[above] {};

					      \draw (1,0pt) -- (1,-2pt) node[anchor=north] {$1$};
					      \draw (2,0pt) -- (2,-2pt) node[anchor=north] {$2$};
					      \draw (4,0pt) -- (4,-2pt) node[anchor=north] {$p-1$};
					      \draw (5,0pt) -- (5,-2pt) node[anchor=north] {$p$};
					      \draw (0pt,2) -- (-2pt,2) node[anchor=east] {$1$};
					      \draw (0pt,2.5) -- (-2pt,2.5) node[anchor=east] {$\frac{p}{p-1}$};
					      \fill[gray!70] (1,2) circle (0.05);
					      \fill[gray!70] (2,2) circle (0.05);

					      \fill (5,2.5) circle (0.05);
					      \draw[line width=0.7pt] (0,0) -- (5,2.5);
					      \draw[dotted] (0,2.5) -- (5,2.5);
					      \draw[dotted] (0,2) -- (4,2);
					      \draw[dotted] (5,0) -- (5,2.5);
					      \draw[dotted] (4,0) -- (4,2);
					      \fill[red] (4,2) circle (0.05);
				      \end{tikzpicture}
				      \caption{$\Newt{\Phi_{\omega_0}}$,\\if $k_{\omega_0}=\frac{1}{p-1}$}
			      \end{minipage}
			      \begin{minipage}[t]{0.49\textwidth}
				      \centering
				      \begin{tikzpicture}[scale=0.8]
					      \draw[->] (0,0) -- (5.5,0) node[right] {};
					      \draw[->] (0,0) -- (0,5.5) node[above] {};

					      \draw (1,0pt) -- (1,-2pt) node[anchor=north] {$1$};
					      \draw (2,0pt) -- (2,-2pt) node[anchor=north] {$2$};
					      \draw (4,0pt) -- (4,-2pt) node[anchor=north] {$p-1$};
					      \draw (5,0pt) -- (5,-2pt) node[anchor=north] {$p$};
					      \draw (0pt,2) -- (-2pt,2) node[anchor=east] {$1$};

					      \draw (0pt,5) -- (-2pt,5) node[anchor=east] {$v_p(\beta_\omega^p-\alpha)$};
					      \fill[gray!70] (1,2) circle (0.05);
					      \fill[gray!70] (2,2) circle (0.05);

					      \fill (5,5) circle (0.05);
					      \draw[line width=0.7pt] (0,0) -- (4,2) -- (5,5);

					      \draw[dotted] (0,5) -- (5,5);
					      \draw[dotted] (0,2) -- (4,2);
					      \draw[dotted] (5,0) -- (5,5);
					      \draw[dotted] (4,0) -- (4,2);
					      \fill (4,2) circle (0.05);
				      \end{tikzpicture}
				      \caption{$\Newt{\Phi_{\omega_0}}$,\\if $k_{\omega_0}>\frac{1}{p-1}$}
			      \end{minipage}
		      \end{figure}
		\item For the case of $\omega>\omega_0$, we have $k_\omega>\frac{1}{p-1}$. With the same calculation as above, one can prove by transfinite induction that for any ordinal $\omega\geq \omega_0+1$, $\beta_\omega\in\Lpha(\frakT_\alpha,\frakF_{\beta_{\omega_0+1}})$.
	\end{enumerate}
	The result follows.
\end{proof}
Additionally, we need the following auxiliary lemma:
\begin{lemma}\label[lemma]{lem:2580}
	For any $p^2$-th primitive root of unity $\zeta_{p^2}$, there exists another $p^2$-th primitive root of unity $\zeta_{p^2}^\prime$ and a $p$-th root of unity $\xi_c$ (not necessarily primitive) that $\zeta_{p^2}=\zeta_{p^2}^\prime\cdot \xi_c$ and $C_{\frac{1}{p-1}}(\zeta_{p^2}^\prime)=0$.
\end{lemma}
\begin{proof}
	Fix a $2(p-1)$-th primitive root of unity $\tilde{\zeta}_{2(p-1)}$. Let $$\calW\coloneqq\left\{\tilde{\zeta}_{2(p-1)}^{2k+1}\colon k\in\bbN_{<p-1}\right\}.$$  By choosing $\zeta_{2(p-1)}$ in the expansion of the $p^2$-th primitive root of unity given by \Cref{eg:29565} (see also \cite[Theorem 3.3]{WangYuan2021}) in $\calW$, we get $p-1$ different $p^2$-th primitive roots of unity $r_0,r_1,\cdots,r_{p-2}$, satisfying $[C_{\frac{1}{p(p-1)}}(r_k)]=\tilde{\zeta}_{2(p-1)}^{2k+1}$ and $[C_{\frac{1}{p-1}}(r_k)]=0$ for every $k\in\bbN_{<p-1}$.

	Similarly, for every $c\in \{0\}\cup\calW$, there exists a $p$-th root of unity (not necessarily primitive) $\xi_c$ that $v_p\left(\xi_c-1-c\cdot p^{\frac{1}{p-1}}\right)> \frac{1}{p-1}$. Thus, for any $k\in\bbN_{<p-1}$ and $c\in\{0\}\cup\calW$, $r_k\cdot \xi_c$ is a $p^2$-th primitive root of unity, satisfying $[C_{\frac{1}{p(p-1)}}(r_k\cdot\xi_c)]=\tilde{\zeta}_{2(p-1)}^{2k+1}$ and $[C_{\frac{1}{p-1}}(r_k\cdot\xi_c)]=c$. This enumerates all $p(p-1)$ $p^2$-th primitive roots of unity. The result follows.
\end{proof}
\begin{proof}[Proof of \Cref{prop:14697}]
	The case of $n=1$ follows immediately from \cite[Proposition 3.4]{WangYuan2021}.

	Let $\zeta_{p^2}$ be any $p^2$-th primitive root of unity. By \Cref{lem:2580}, there exists another $p^2$-th primitive root of unity $\zeta_{p^2}^\prime$ and a $p$-th root of unity $\xi_c$ (not necessarily primitive) that $\zeta_{p^2}^p=\zeta_{p^2}^\prime\cdot \xi_c$ and $C_{\frac{1}{p-1}}(\zeta_{p^2}^\prime)=0$. By applying \Cref{lem:35078}, we have
	$$\zeta_{p^2}^\prime\in\Lpha(\frakT_{(\zeta_{p^2}^\prime)^p},\frakF_{(\zeta_{p^2}^\prime)^p})=\Lpha(p-1,2).$$
	Since $\xi_c\in\Lpha(p-1,2)$, we know that $\zeta_{p^2}\in\Lpha(p-1,2)$. On the other hand, by \cite[Theorem 3.3]{WangYuan2021}, one has $\frakT_{\zeta_{p^2}}\geq p-1$ and $\frakF_{\zeta_{p^2}}\geq 2$. This implies that $\frakT_{\zeta_{p^2}}=p-1$ and $\frakF_{\zeta_{p^2}}=2$.

	When $n\geq 3$, we can set $\alpha=\left(\zeta_{p^n}\right)^p$ in \Cref{lem:35078} inductively to get the result. One should notice that when $n\geq 3$, we no longer know if the analog of \Cref{lem:2580} holds for $\zeta_{p^n}$. Thus, the hyper-inertia index is multiplied by $p$ when $n$ increases by $1$.
\end{proof}

\begin{corollary}\label[corollary]{cor:6370}
	For any positive integer $m=r\cdot p^{v_p(m)}$ with $\gcd(r,p)=1$ and any $m$-th primitive root of unity $\zeta_m$, one has
	\begin{enumerate}
		\item If $v_p(m)=0$, then $\frakT_{\zeta_m}=1$ and $\frakF_{\zeta_m}=\opn{ord}_r p$.
		\item If $v_p(m)\geq 1$, then $\frakT_{\zeta_m}\mid p-1$ and
		      $$\frakF_{\zeta_m}\mid\begin{cases}
				      \opn{lcm}(2,\opn{ord}_rp),                              & \text{ if }v_p(m)=1,2;   \\
				      \opn{lcm}\left(2\cdot p^{v_p(m)-1},\opn{ord}_rp\right), & \text{ if }v_p(m)\geq 3.
			      \end{cases}$$
	\end{enumerate}
\end{corollary}
\begin{proof}
	It suffices to note that any $r$-th root of unity lies in $W(\bbF_{p^{\opn{ord}_rp}})$.
\end{proof}

With the power of the local Kronecker-Weber theorem, we can generalize this result to those $p$-adic algebraic numbers that generate abelian extensions over $\bbQ_p$:
\begin{theorem}\label{thm:53765}
	Let $\alpha\in \overline{\bbQ}_p$ be a $p$-adic algebraic number with $\bbQ_p(\alpha)/
		\bbQ_p$ an abelian extension of degree $n$. Denote by $\bff_{\bbQ_p(\alpha)}$ the local conductor of $\bbQ_p(\alpha)$ over $\bbQ_p$. Then
	\begin{enumerate}
		\item If $\bff_{\bbQ_p(\alpha)}=0$, then $\frakT_\alpha=1$ and $\frakF_\alpha=n$.
		\item If $\bff_{\bbQ_p(\alpha)}\geq 1$, then $\frakT_\alpha\mid p-1$ and
		      $$\frakF_\alpha\mid\begin{cases}
				      \opn{lcm}(2,n),                                              & \text{ if }\bff_{\bbQ_p(\alpha)}=1,2;   \\
				      \opn{lcm}\left(2\cdot p^{ \bff_{\bbQ_p(\alpha)}-1},n\right), & \text{ if }\bff_{\bbQ_p(\alpha)}\geq 3.
			      \end{cases}.$$
	\end{enumerate}
\end{theorem}
To prove this theorem, the following effective form of the local Kronecker-Weber theorem is needed:
\begin{lemma}\label[lemma]{lem:15272}
	Let $K/\bbQ_p$ be an abelian extension of degree $n$ with conductor $\bff_K$ and let $m=\left(p^{n}-1\right)p^{\bff_K}$. Then $K\subseteq \bbQ_p(\zeta_m)$.
\end{lemma}
\begin{proof}
	By \cite[Lemma 4.11]{guillotGentleCourseLocal2018} and its proof, there exists $s\geq 1$ that
	\[\langle p^s\rangle\times U_{\bbQ_p}^{(\bff_K)}\subseteq \calN_{K/\bbQ_p}K^\times.\] It follows that $K\subseteq\bbQ_p\left(\zeta_{(p^s-1)p^{\bff_K}}\right)$ by the proof of \cite[Theorem 13.27]{guillotGentleCourseLocal2018}. On the other hand, we have $K\subseteq\bbQ_p\left(\zeta_{(p^n-1)p^{v_p(n)+2}}\right)$ by  \cite[Theorem 3.1]{koenigsmann2022elementary}. Since
	$$\bbQ_p\left(\zeta_{(p^s-1)p^{\bff_K}}\right)\cap\bbQ_p\left(\zeta_{(p^n-1)p^{v_p(n)+2}}\right)\subseteq \bbQ_p(\zeta_m),$$
	we have $K\subseteq \bbQ_p(\zeta_m)$.
\end{proof}
\begin{proof}[Proof of \Cref{thm:53765}]
	Let $m=\left(p^{n}-1\right)p^{\bff_{\bbQ_p(\alpha)}}$.
	By \Cref{lem:15272}, we know that $\alpha\in\bbQ_p(\zeta_m)$.

	Note $\opn{ord}_{p^{n}-1}p=n$. By \Cref{cor:6370}, we know that
	$$\frakT_{\zeta_m}=\begin{cases}
			1,   & \text{ if }\bff_{\bbQ_p(\alpha)}=0;     \\
			p-1, & \text{ if }\bff_{\bbQ_p(\alpha)}\geq 1,
		\end{cases}$$
	and
	$$\frakF_{\zeta_m}\begin{cases}\begin{alignedat}{3}
				 & \omit\hfill $=$ \hfill &  & n,                                                           &  & \text{ if }\bff_{\bbQ_p(\alpha)}=0;     \\
				 & \omit\hfill $=$ \hfill &  & \opn{lcm}(2,n),                                              &  & \text{ if }\bff_{\bbQ_p(\alpha)}=1,2;   \\
				 & \text{divides }        &  & \opn{lcm}\left(2\cdot p^{ \bff_{\bbQ_p(\alpha)}-1},n\right), &  & \text{ if }\bff_{\bbQ_p(\alpha)}\geq 3.
			\end{alignedat}
		\end{cases}$$
	Since $\alpha\in\bbQ_p(\zeta_m)\subseteq \Lpha(\frakT_{\zeta_m},\frakF_{\zeta_m})$, the result follows.
\end{proof}

\subsection{Criterion for tamely ramified extensions}

\begin{theorem}\leavevmode\label{thm:54918new}Let $\alpha\in\Lpha$ be a hyper-algebraic element in $\bbL_p$. Then $\bbQ_p(\alpha)$ is tamely ramified over $\bbQ_p$ if and only if $\Supp(\alpha)\subseteq\frac{1}{\frakT_\alpha}\bbZ$. In this situation, we have $\frakT_\alpha=\frake_\alpha$, $\frakf_\alpha\mid \frakF_\alpha$ and $\frakF_\alpha \mid c$, where $c\coloneqq \opn{ord}_{\opn{lcm}(\frake_\alpha,p^{\frakf_\alpha}-1)}p$ and $\frakf_\alpha$ (resp. $\frake_\alpha$) is the inertia degree (resp. the ramification index) of the extension $\bbQ_p(\alpha)/\bbQ_p$.
\end{theorem}
The proof of this theorem is based on the subsequent lemma, which is generally standard in local field theory (cf. \cite[Proposition II.5.12]{langAlgebraicNumberTheory1994}), albeit presented here with slightly greater explicitness.
\begin{lemma}\label[lemma]{lem:57300new}
	Let $\alpha\in\overline{\bbQ}_p$ be a $p$-adic algebraic number with $\bbQ_p(\alpha)$ tamely ramified over $\bbQ_p$. Then there exists a $\frake_\alpha$-th root $\zeta_e\in\overline{\bbF}_p$ of unity that
	$$\bbQ_p(\alpha)=\bbQ_{p^{\frakf_\alpha}}\left(p^{1/\frake_\alpha}\cdot[\zeta_e]\right),$$
	where $\bbQ_{p^{\frakf_\alpha}}\coloneqq W(\bbF_{p^{\frakf_\alpha}})\left[\frac{1}{p}\right]$ is the maximal unramified extension of $\bbQ_p$ in $\bbQ_p(\alpha)$.
\end{lemma}
\begin{proof}
	Let $\calO_K$ be the ring of integer of $K\coloneqq\bbQ_p(\alpha)$ with a uniformizer $\pi_K$. Suppose $\pi_K^{\frake_\alpha}=p\cdot u$, where $u$ is a unit in $\calO_K^\times$.

	Note that the polynomial $T^{\frake_\alpha}-\overline{u}\in\bbF_{p^{\frakf_\alpha}}[T]$ has simple roots by the condition $\gcd(\frake_\alpha,p)=1$. Hensel lemma implies that there is a $\frake_\alpha$-th root $v$ of $u$ in $\calO_K^\times$. If we set $\pi_K^\prime\coloneqq \pi_K\cdot v^{-1}$, then this element is also a uniformizer of $K$. Since $\pi_K^\prime$ is a $\frake_\alpha$-th root of $p$, we have $\pi_K^\prime=p^{1/\frake_\alpha}\cdot[\zeta_e]$ for some $\frake_\alpha$-th root $\zeta_e$ of unity in $\overline{\bbF}_p$.%

\end{proof}

\begin{proof}[Proof of \Cref{thm:54918new}]
	If $\opn{supp}(\alpha)\subseteq \frac{1}{\frakT_\alpha}\bbZ$, we can write
	$\alpha=\sum_{k>\!\!\!>-\infty}^{+\infty}[r_k]\cdot p^{\frac{k}{\frakT_\alpha}}$,
	where $r_k\in\bbF_{p^{\frakF_\alpha}}$ for all $k$.
	Thus, $\alpha$ lies in $\bbQ_{p^{\frakF_\alpha}}\left(p^{\frac{1}{\frakT_\alpha}}\right)$, where $\bbQ_{p^{\frakF_\alpha}}\coloneqq W(\bbF_{p^{\frakF_\alpha}})\left[\frac{1}{p}\right]$ is the unique unramified extension of $\bbQ_p$ with residue field $\bbF_{p^{\frakF_\alpha}}$. Since $\frakT_\alpha$ is coprime to $p$ (cf. \Cref{lem:40178}), the field $\bbQ_{p^{\frakF_\alpha}}\left(p^{\frac{1}{\frakT_\alpha}}\right)$ is tamely ramified over $\bbQ_p$, implying that $\bbQ_p(\alpha)$ is also tamely ramified over $\bbQ_p$.

	Conversely, if $\bbQ_p(\alpha)/\bbQ_p$ is tamely ramified, then we have $$\bbQ_p(\alpha)=\bbQ_{p^{\frakf_\alpha}}\left(p^{1/\frake_\alpha}\cdot[\zeta_e]\right)$$
	for some $\frake_\alpha$-th root $\zeta_e\in\overline{\bbF}_p$ of unity by \Cref{lem:57300new}. Let  $$\alpha=\sum_{k=0}^{\frake_\alpha-1}c_k\cdot\left(p^{1/\frake_\alpha}\cdot[\zeta_e]\right)^k$$
	with $c_k\in \bbQ_{p^{\frakf_\alpha}}$ for $k=0,\cdots,\frake_\alpha-1$. If we set $c_k=\sum_{i>-\infty}\left[c_i^{(k)}\right]p^i\in\bbQ_{p^{\frakf_\alpha}}$ with $c_i^{(k)}\in\bbF_{p^{\frakf_\alpha}}$,
	then
	\begin{equation}\label{eq:56384}
		\alpha=\sum_{k=0}^{\frake_\alpha-1}\sum_{i>-\infty}\left[c_i^{(k)}\cdot\zeta_e^k\right] p^{i+k/\frake_\alpha}.
	\end{equation}
	This shows that
	$\Supp(\alpha)\subseteq \frac{1}{\frake_\alpha}\bbZ$. Thus, $$\Supp(\alpha)\subseteq \frac{1}{\frake_\alpha}\bbZ\cap \frac{1}{\frakT_\alpha}\bbZ[1/p]\subseteq \bbZ_{(p)}\cap\frac{1}{\frakT_\alpha}\bbZ[1/p]=\frac{1}{\frakT_\alpha}\bbZ.$$

	To prove the second assertion, notice that the inclusion $\alpha\in\bbQ_{p^{\frakF_\alpha}}\left(p^{\frac{1}{\frakT_\alpha}}\right)$ implies $\frake_\alpha\mid \frakT_\alpha$ and $\frakf_\alpha\mid\frakF_\alpha$. On the other hand, if any coefficient $c_i^{(k)}\cdot \zeta_e^k$ in \eqref{eq:56384} is non-zero, then it is a $\opn{lcm}(\frake_\alpha,p^{\frakf_\alpha}-1)$-th root of unity, i.e. $c_i^{(k)}\cdot \zeta_e^k\in\bbF_{p^c}$. As a result, one conclude by \Cref{lem:40178} that $\alpha\in\Lpha(\frake_\alpha,c)$.
\end{proof}
Compared to \Cref{prop:47781}, the constant $c$ in \Cref{thm:54918new} does provide a better bound for the hyper-inertia index in the tamely ramified case:%
\begin{lemma}
	Let $c\coloneqq \opn{ord}_{\opn{lcm}(\frake_\alpha,p^{\frakf_\alpha}-1)}p$ be the constant in \Cref{thm:54918new}. Then $c$ divides $\opn{lcm}(\phi(\frake_\alpha),\frakf_\alpha)$, where $\phi$ is Euler's totient function.
\end{lemma}
\begin{proof}
	Let $e_0\coloneqq\frac{\frake_\alpha}{\gcd(\frake_\alpha,p^{\frakf_\alpha}-1)}$, then $\opn{lcm}(\frake_\alpha,p^{\frakf_\alpha}-1)=e_0\cdot (p^{\frakf_\alpha}-1)$, with $e_0$ a factor of $\frake_\alpha$ that coprime to $p^{\frakf_\alpha}-1$ and $p$. Chinese remainder theorem implies that
	$$c=\opn{lcm}(\opn{ord}_{e_0}p,\opn{ord}_{p^{\frakf_\alpha}-1}p)=\opn{lcm}(\opn{ord}_{e_0}p,\frakf_\alpha).$$
	Since $e_0$ is a factor of $\frake_\alpha$, we have $\opn{ord}_{e_0}p$ divides $\opn{ord}_{\frake_\alpha}p$. The result follows from Euler's theorem.
\end{proof}
\printbibliography

\end{document}